\documentclass[11pt, epsfig]{article}
\usepackage{epsfig, amsmath, amssymb, amsthm, times, stackrel}
\newtheorem {thm}{Theorem}[section]
\newtheorem {lem}[thm]{Lemma}

\theoremstyle{defintion}
\newtheorem {df}[thm]{Definition}

\theoremstyle{remark}
\newtheorem{rem}[thm]{Remark}

\theoremstyle{example}
\newtheorem{ex}[thm]{Example}

\def\p{\partial}

\def\eps{\epsilon}

\def\R{{\mathbb R}}
\def\Z{{\mathbb Z}}
\def\N{{\mathbb N}}

\def\SS{{\mathbb S}}
\def\T{\mathcal{T}}
\def\1{\operatorname{\mathbf 1}}

\def\lbl{\label}
\def\be{\begin{equation}}
\def\ee{\end{equation}}
\def\lbl{\label}

\def\t{\mathsf{T}}

\def\E{{\mathbb E}}
\def\P{{\mathbb P}}
\def\M{{\mathcal M}}

\title{The mean field equation for the Kuramoto model on graph
  sequences with non-Lipschitz limit}
\author{Dmitry Kaliuzhnyi-Verbovetskyi and Georgi S. Medvedev
\thanks{
Department of Mathematics, Drexel University, 3141 Chestnut Street,
Philadelphia, PA 19104;
{\tt dmitryk@math.drexel.edu},
{\tt medvedev@drexel.edu}
}
}

\begin{document}
\maketitle
\begin{abstract}
The Kuramoto model (KM) of coupled phase oscillators on graphs provides
the most influential framework for studying collective dynamics and 
synchronization. It exhibits a rich repertoire of dynamical regimes. Since the work
of Strogatz and Mirollo \cite{StrMir91}, the mean field equation derived in the 
limit as the number of oscillators in the KM goes to infinity, has been the 
key to understanding a number of interesting effects, including the onset of 
synchronization and 
chimera states. In this work, we study the mathematical basis of
the mean field equation as an approximation of the discrete KM. Specifically, we 
extend the Neunzert's method of rigorous justification of the mean field equation
(cf.~\cite{Neu78})
to cover interacting dynamical systems on graphs. We then apply it to the 
KM on convergent graph sequences with non-Lipschitz limit. 
This family of graphs includes many graphs that are of interest
in applications, e.g., nearest-neighbor and small-world graphs. 
The approaches for justifying the 
mean field limit for the KM  proposed previously in 
\cite{Lan05, ChiMed16} do not cover the non-Lipschitz case. 
\end{abstract}

\section{Introduction}
\setcounter{equation}{0}

The KM of coupled phase oscillators provides a useful framework for
studying collective behavior in large ensembles of interacting
dynamical systems. It is derived from a weakly coupled system of nonlinear 
oscillators, which are described by autonomous systems of ordinary differential
equations possessing a stable limit cycle \cite{HopIzh97}. Originally, Kuramoto considered
all-to-all coupled systems, in which each oscillator interacts with all other oscillators 
in exactly the same way. In this case, the KM has the following form
\be\lbl{classKM}
\dot u_{ni} =\omega_i + {K\over n} \sum_{j=1}^n 
\sin \left( u_{nj}-u_{ni}+\alpha\right),\quad
i\in [n]:=\{1,2,\dots,n\}.
\ee
Here, $u_{ni}:\R^+\to \SS:=\R/2\pi\Z$ stands for the phase of oscillator~$i$ as a function of time,
$\omega_i$ is its intrinsic frequency, $K$ is the coupling strength, and $\alpha$ is a parameter
defining the type of interactions. 

Despite its simple form, the KM \eqref{classKM}
features a rich repertoire of interesting dynamical effects. For the purpose of this review,
we mention the onset of synchronization in \eqref{classKM} with randomly distributed intrinsic
frequencies $\omega_i$ (Fig.~\ref{f.1}\textbf{a}, \textbf{b})
(cf.~\cite{Str00}) and chimera states, 
interesting
spatio-temporal patterns combining coherent and incoherent behaviors
% (Fig.~\ref{f.1} \textbf{c}) 
\cite{KurBat02, AbrStr06}. 
The mathematical analysis of these and many other dynamical regimes uses
the mean field equation, derived in the limit when the number of oscillators 
goes to infinity \cite{StrMir91}. 
% For the KM, this approach was pioneered by Strogatz and Mirrolo
% in their work on synchronization \cite{StrMir91}. 
The mean field equation is a partial
differential equation for the probability density describing the distribution of the
phases on $\SS$. We discuss the mean field equation in more detail below.

Recently, there has been a growing interest in the dynamics of coupled dynamical 
systems on graphs \cite{PG16}.
In the KM on a graph, each oscillator is placed at a node of an undirected graph 
$\Gamma_n=\langle V(\Gamma_n), E(\Gamma_n)\rangle$.
Here, $V(\Gamma_n)=[n]$ stands for the node set of $\Gamma_n,$ and
$E(\Gamma_n)$ denotes its edge set.
The oscillator~$i$ interacts only with the oscillators at the adjacent nodes:
\be\lbl{KMonG}
\dot u_{ni} =\omega_i + {K\over n} \sum_{j:j\sim i} 
\sin \left( u_{nj}-u_{ni}+\alpha\right),\quad
i\in [n],
\ee
where $j\sim i$ is a shorthand for $\{i,j\}\in E(\Gamma_n)$.

Clearly, one can not expect  limiting
behavior of solutions of \eqref{KMonG} as $n\to\infty$, unless 
the graph sequence $\{\Gamma_n\}$ is convergent in the 
appropriate sense. In the present paper, we use the following construction of the 
convergent sequence $\{\Gamma_n\}$. Let $W$ be a symmetric measurable 
function on the unit square $I^2:=[0,1]$. $W$ is called a graphon. It  will be used to 
define the asymptotic behavior of $\{\Gamma_n\}$. Further, let
\be\lbl{Xn}
X_n=\{x_{n1}, x_{n2},\dots, x_{nn}\},\quad x_{ni}=i/n, \; i\in [n],
\ee
and
\be\lbl{weights}
W_{n,ij}:=n^2\int_{I_{n,i}\times I_{n,j} } W(x,y) dxdy, \quad 
I_{n,i}:=[x_{n(i-1)}, x_{ni}), \; i,j \in [n].
\ee
The weighted graph $\Gamma_n=G(W,X_n)$ on $n$ nodes is defined as follows.
The vertex set is $V(\Gamma_n)=[n]$ and the edge set is
\be\lbl{edge-set}
E(\Gamma_n)=\left\{ \{i,j\}:\;  W_{n,ij}\neq 0,\; i,j\in [n]\right\}.
\ee
Each edge $\{i,j\}\in E(\Gamma_n)$ is supplied with the weight
$W_{n,ij}$\footnote{There are several possible ways of defining the weights $W_{n,ij}, i,j\in [n]$
(see Remark~\ref{rem.equiv}).}.
\begin{figure}
\begin{center}
\textbf{a}\includegraphics[height=1.8in,width=2.0in]{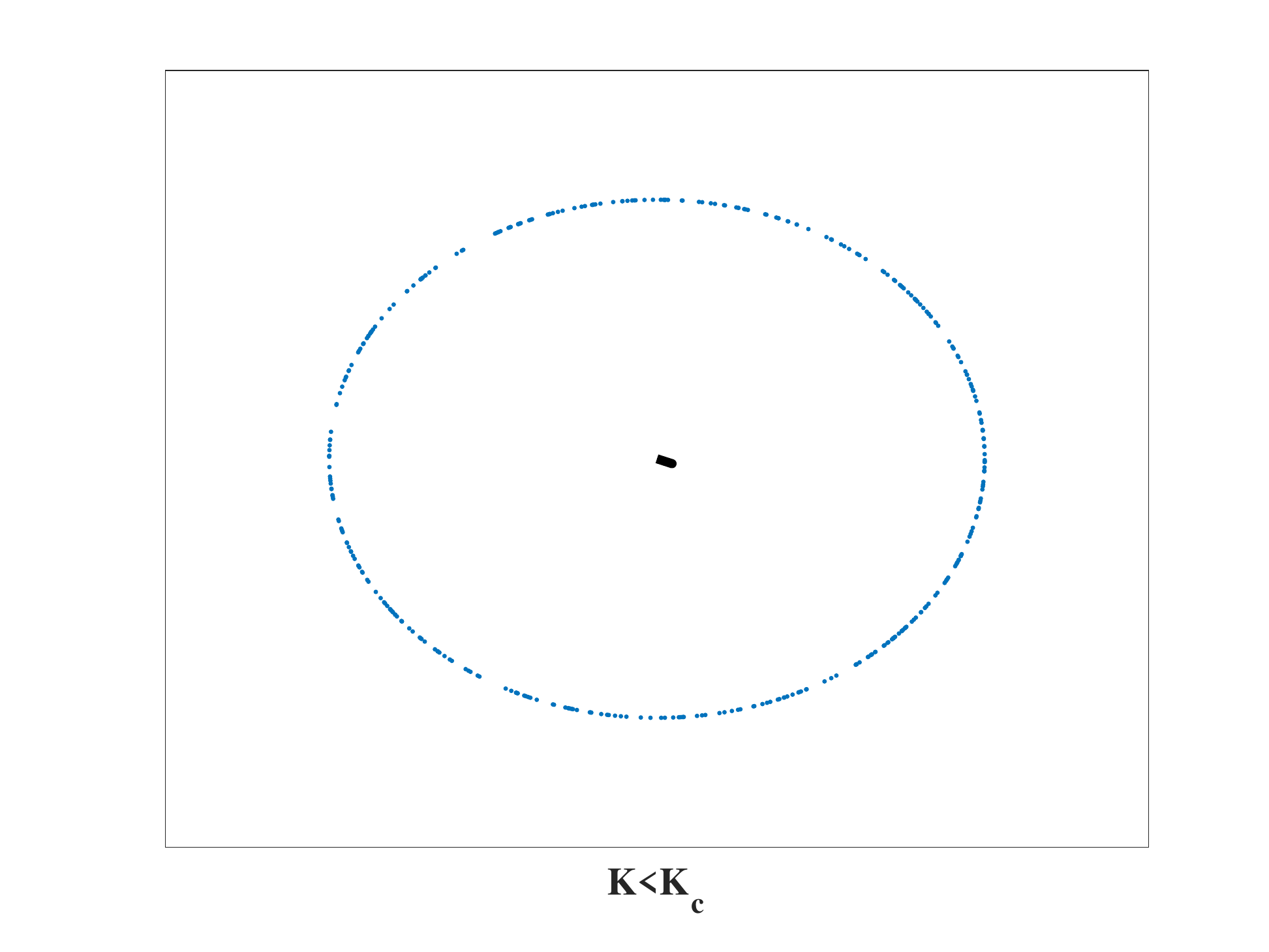}
\textbf{b}\includegraphics[height=1.8in,width=2.0in]{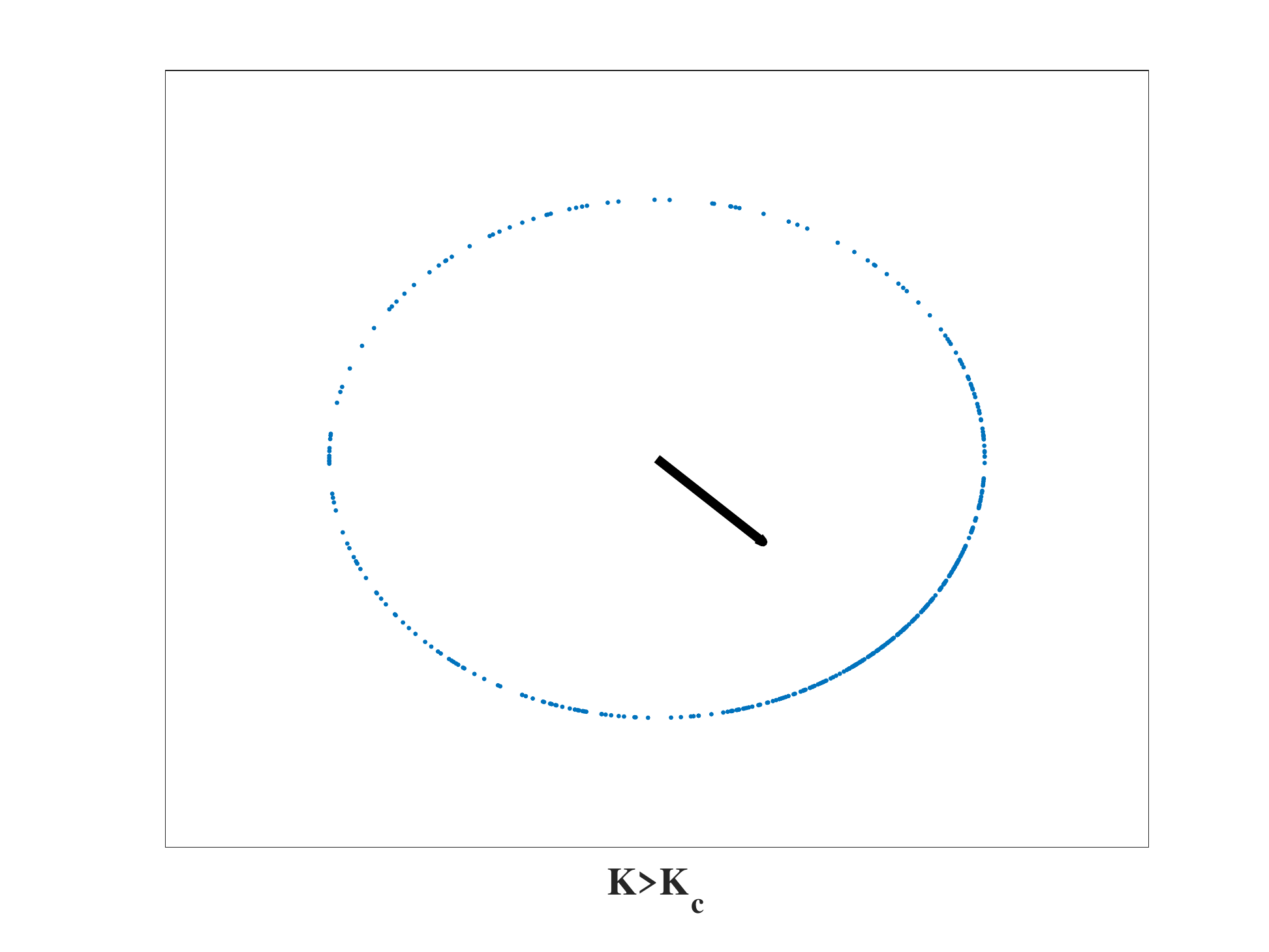}
% \textbf{c}\includegraphics[height=1.8in,width=2.0in]{ff1a.pdf}
\end{center}
\caption{
The distribution of the phase oscillators in the KM \eqref{classKM}
for values of $K$ below (\textbf{a}) and above (\textbf{b}) the critical value
$K_c$. In the former plot, the distribution is approximately uniform, whereas
 the latter plot exhibits a pronounced cluster. The bold vectors depict the 
order parameter, whose length reflects the degree of coherence. 
The random distribution of the oscillators shown in these plots can 
be effectively analyzed with the mean field equation \eqref{MF}. In particular,
the mean field analysis determines 
the critical
value $K_c$ is determined from the mean field equation. 
% Plot (\textbf{c})
% shows a stable steady state solution of the KM on the power law graph. The 
% phases are distributed randomly. The density is determined from the mean field
% equation \eqref{MF}.
}
\lbl{f.1}
\end{figure}

The KM on $\Gamma_n=G(W,X_n)$ has the following form
\be\lbl{wKM}
\dot u_{ni} =\omega_i +{K\over n} \sum_{j=1}^n W_{n,ij} \sin\left( u_{nj}-u_{ni}+\alpha\right),
\quad i\in [n].
\ee
For different $W$ \eqref{wKM} implements the KM on a variety of simple and weighted graphs.
Moreover, it provides an effective approximation of the KM on random graphs.
Indeed, let $\bar \Gamma_n=G_r(X_n,W)$ be a random graph on $n$ nodes, whose edge set is 
defined as follows:
\be\lbl{Pedge}
\P\left( \{i,j\}\in E(\Gamma_n)\right)= W_{n,ij},
\ee
assuming the range of $W$ is $[0,1]$.
The decision for each pair $\{i,j\}$ is made independently from the decisions on other pairs.
$\bar \Gamma_n=G_r(X_n,W)$ is called a W-random graph \cite{LovSze06}.

The KM on the W-random graph $\bar\Gamma_n=G_r(X_n,W)$ has the following form:
\be\lbl{rwKM}
\dot{\bar u}_{ni} = \omega_{i} + Kn^{-1} \sum_{j=1}^n e_{nij} \sin(\bar u_{nj}-\bar u_{ni}),
\quad i\in [n],
\ee
where  $e_{nij}, 1\le i\le j\le n$ are independent Bernoulli RVs:
$$
\P(e_{nij}=1)=W_{n,ij},
$$
and $e_{nij}=e_{nji}.$ 

The following lemma shows that the deterministic model \eqref{wKM} approximates
the KM on the random graph ${\bar \Gamma}_n$ \eqref{rwKM}.
\begin{lem}\lbl{lem.ave} \cite{ChiMed16}
Let $u_n(t)$ and $\bar u_n(t)$ denote solutions of the IVP for
\eqref{wKM} and \eqref{rwKM} respectively. Suppose that
the initial data for these problems coincide
$
u_n(0)=\bar u_n(0).
$ 
Then 
\be\lbl{Wave}
\lim_{n\to\infty} 
\sup_{t\in [0,T]} \left\|u_n(t)-\bar u_n(t)\right\|_{1,n} =0 \quad \mbox{a.s.},
\ee
where $u_n=(u_1, u_2, \dots, u_n),$ ${\bar u}_n=({\bar u}_1, {\bar u}_2, \dots, {\bar u}_n),$ 
and
\be\lbl{norm-1}
\|u_n \|_{1,n}=\sqrt{n^{-1}\sum_{i=1}^n u_{ni}^2}
\ee
is a discrete $L^2$-norm.
\end{lem}

\begin{ex}\lbl{ex.graphs}
A few examples are in order.
\begin{enumerate}
\item  Let $W(x,y) \equiv p\in (0,1)$. Then 
$\bar \Gamma_n=G_r(X_n,W)$ is an Erd\H{o}s-R{\' e}nyi graph (Fig.~\ref{f.2}a).
\item
 Let 
\be\lbl{def-Wpr}
W_{p,\mathit{h}}(x,y)=
\left\{ 
\begin{array}{ll}
1-p, & d_\SS(2\pi x,2\pi y)\le 2\pi \mathit{h},\\
p, &\mbox{otherwise},
\end{array}
\right.
\ee
where $p, \mathit{h}\in (0,1/2)$ are two parameters
and
\be\lbl{dSS}
d_\SS(\theta,\theta^\prime)=\min \{|\theta-\theta^\prime|, 2\pi-|\theta-\theta^\prime|\}
\ee
is the distance on $\SS$. Then
$\bar \Gamma_n=G_r(X_n,W_{p,\mathit{h}})$ is a W-small-world graph
\cite{Med14b} (Fig.~\ref{f.2}b).
\item
$\Gamma_n=G (X_n,W_{1,\mathit{h}})$ is a $k$-nearest-neighbor graph (Fig.~\ref{f.2}c).
\end{enumerate}
For more examples, we refer an interested reader to \cite{ChiMed16}.
\end{ex}
\begin{figure}
\begin{center}
\textbf{a}\includegraphics[height=1.8in,width=2.0in]{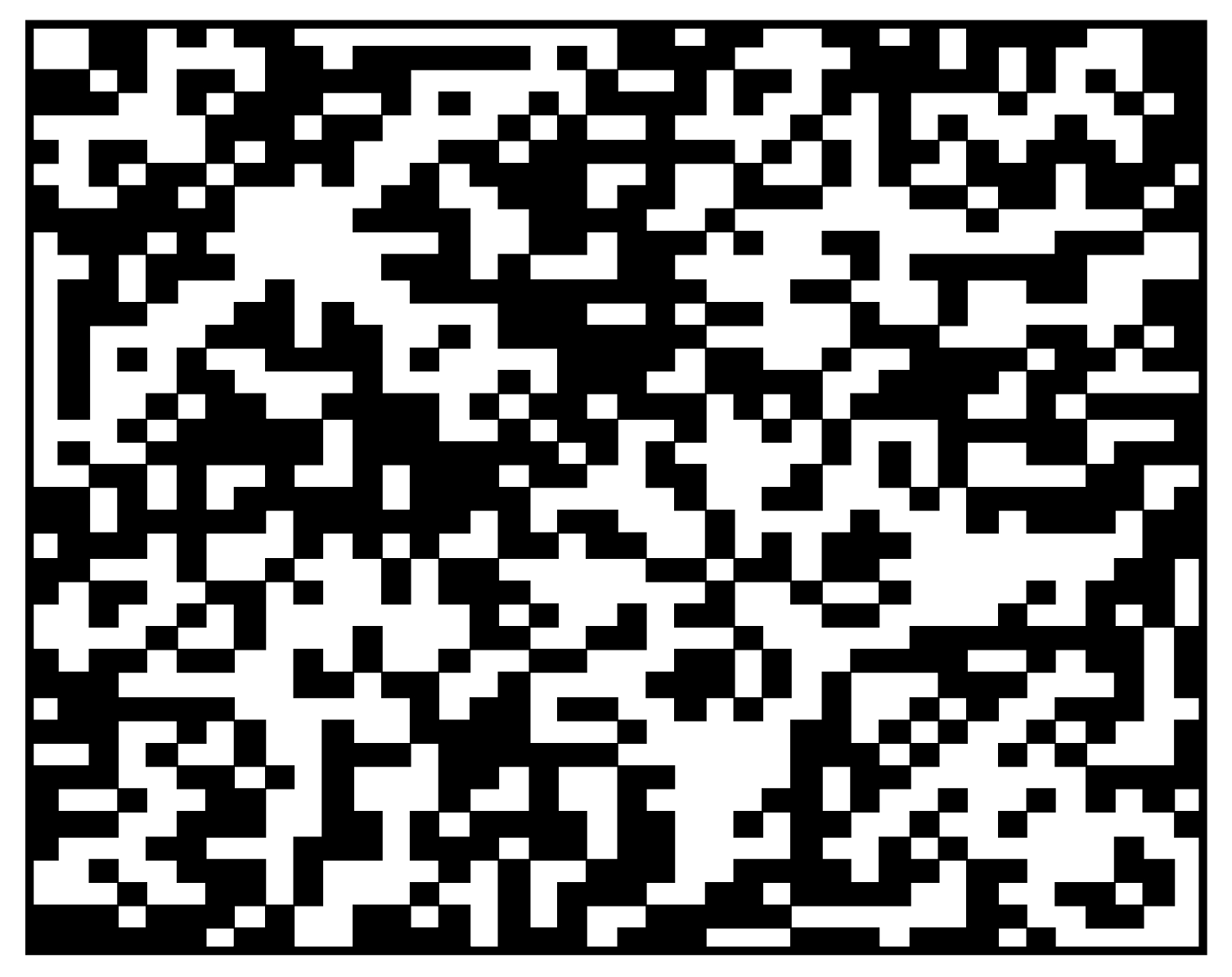}
\textbf{b}\includegraphics[height=1.8in,width=2.0in]{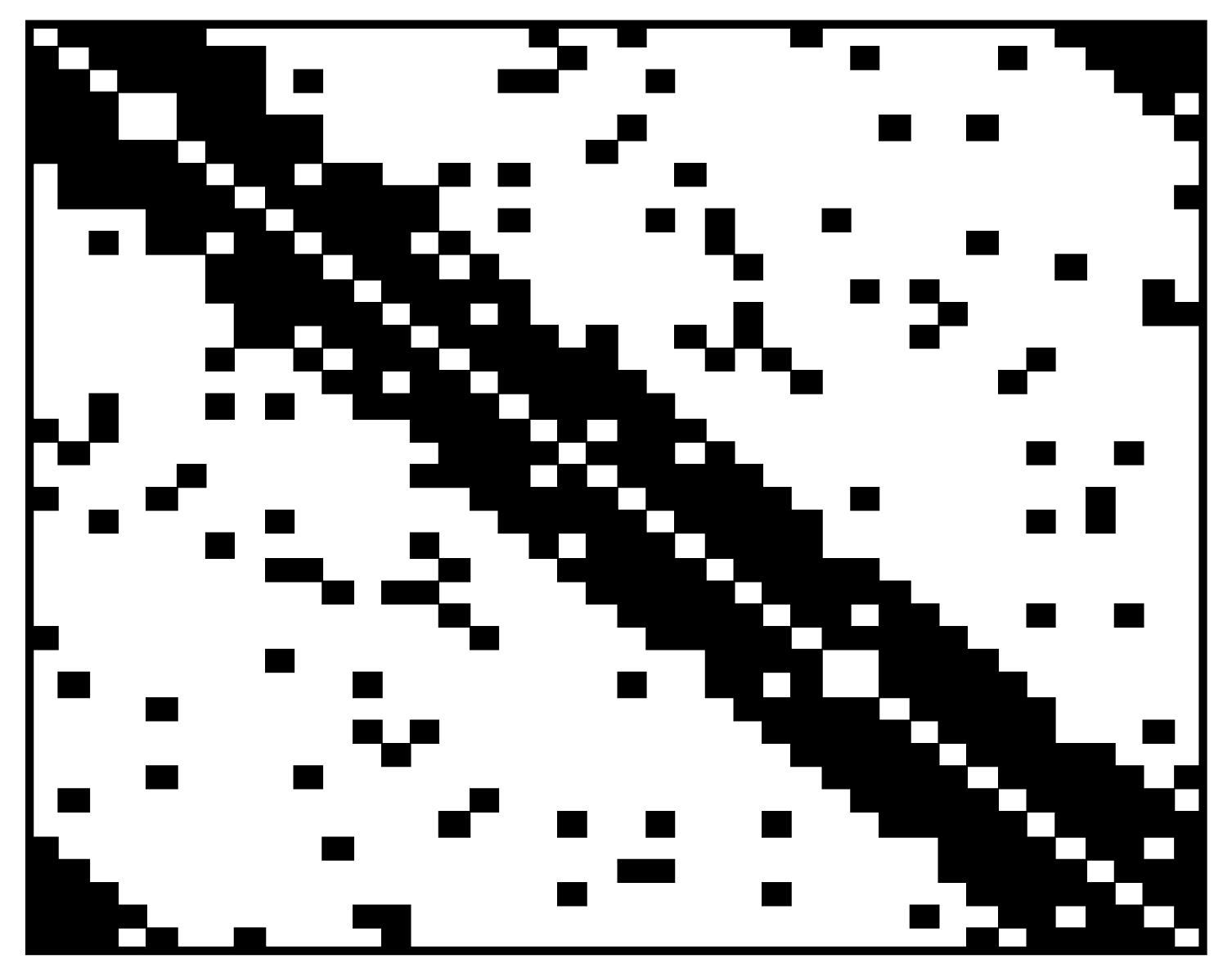}
\textbf{c}\includegraphics[height=1.8in,width=2.0in]{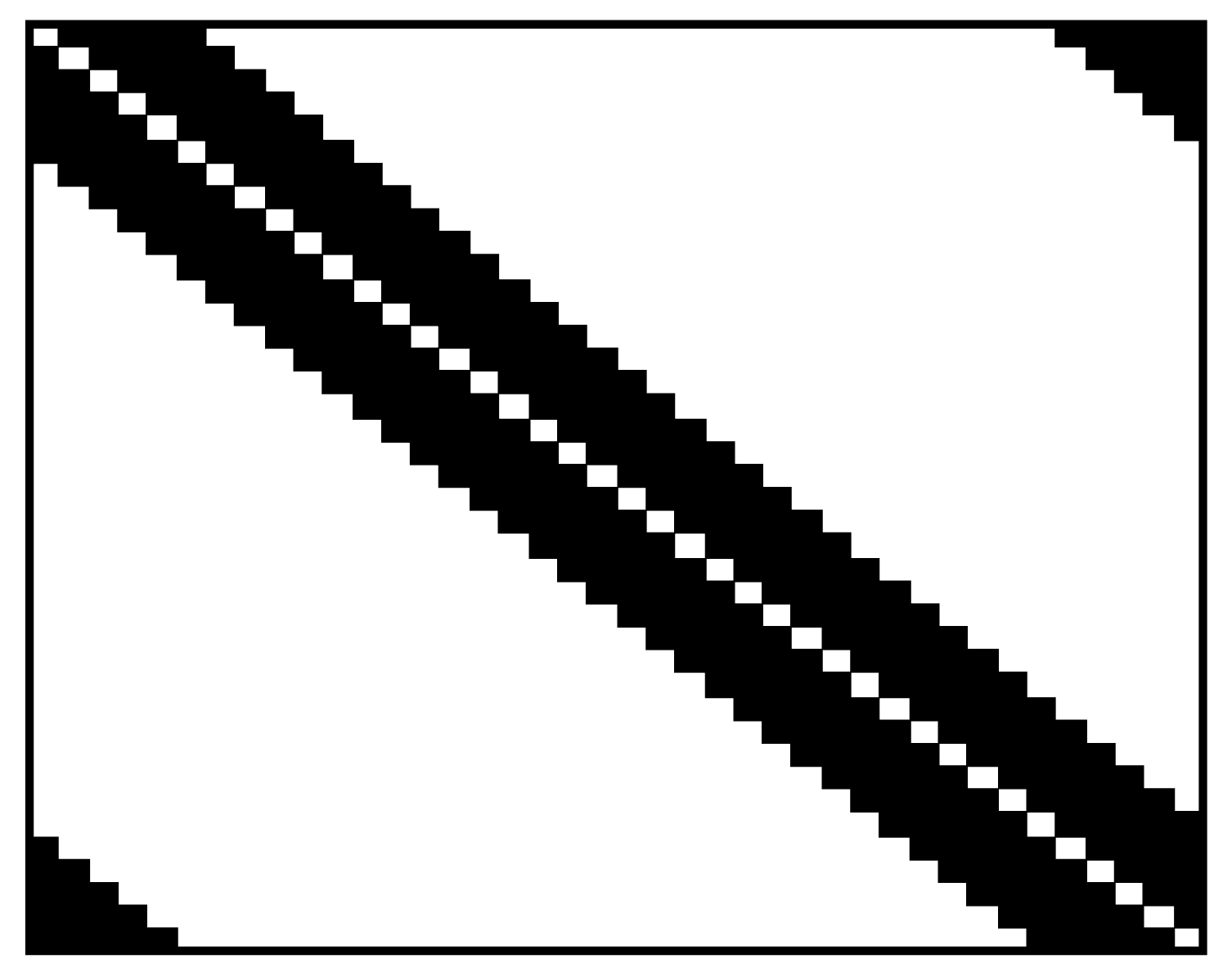}
\end{center}
\caption{ The pixel pictures representing adjacency matrices of the 
Erd\H{o}s-R{\' e}nyi (\textbf{a}), small-world (\textbf{b}), and
nearest-neighbor (\textbf{c}) graphs.
}\lbl{f.2}
\end{figure}
\begin{rem}\lbl{rem.sparse} For simplicity, we restrict the presentation  to the KM
on dense graphs. The KM on W-random graphs \eqref{rwKM} easily extends to sparse
graphs  like  scale-free  graphs (see \cite{KVMed17} for details).
\end{rem}

Below, we will focus on the deterministic
model \eqref{wKM}. All results for this model can be extended to the KM on random graphs
via Lemma~\ref{lem.ave}. Furthermore, from now on we will assume that all intrinsic frequencies
in \eqref{wKM} are the same $\omega_i=\omega$, $i\in [n]$, and, thus, $\omega$ can be
set to $0$ by switching to the rotating frame of coordinates. Extending the analysis in the main
part of this paper to models with distributed frequencies $\omega_i$ is straightforward 
(see, e.g., \cite{ChiMed16}), but it complicates the presentation. We will comment on the 
adjustments in the analysis that are necessary  to cover the distributed intrinsic frequencies 
case in Section~\ref{sec.discuss}.
Until then we consider the following system of $n$ coupled oscillators
on $\Gamma_n=G(W,X_n)$:
\begin{eqnarray}\lbl{KM}
\dot u_{n,i}& = &n^{-1} \sum_{j=1}^n W_{n,ij} D(u_{n,j}-u_{n,i}),\\
\lbl{KM-ic}
u_{n,i}(0)&=& u_{n,i}^0, \quad i\in [n],
\end{eqnarray}
where $D$ is a Lipschitz continuous $2\pi$-periodic function.

Without loss of generality, we assume
\be\lbl{bounds}
\sup_{(x,y)\in I^2} |W(x,y)|\le 1,\quad \max_{u\in\SS} |D(u)|\le 1,
\ee 
and
\be\lbl{Lip-D}
|D(u)-D(v)|\le |u-v|\quad \forall u,v\in \SS.
\ee
In addition, we assume that  the graphon $W$ satisfies the following condition:
\be\lbl{Wcont}
\lim_{\delta\to 0} \int_I |W(x+\delta, y)-W(x,y)|dy =0 \quad \forall x\in I.
\ee

Having defined the KMs on deterministic and random graphs \eqref{wKM} and \eqref{rwKM}
respectively, we will now turn to the mean field limit:
\be\lbl{MF}
{\p\over \p t} \rho(t,u,x) + {\p\over \p u} \left\{ V(t,u,x) \rho(t,u,x)\right\} =0,
\ee
where
\be\lbl{VF}
V(t,u,x)= \int_I\int_\SS W(x,y) D(v-u)\rho(t,v,y)dvdy.
\ee
The initial condition
\be\lbl{MF-ic}
\rho(0,u,x)=\rho^0(u,x) \in L^1(G)
\ee
is a probability density function on $\SS$ for every $x\in I$.
% Thus, $\rho^0(\cdot,x)\in L^1(\SS)$
% for every $x\in I$. 
It approximates the distribution of the initial conditions \eqref{KM-ic}.

In the continuum limit as $n\to\infty$, the nodes of $\Gamma_n$ fill out $I$. Thus, heuristically,
$\rho(t,u,x)$ in \eqref{MF} stands for the density of the probability distribution of the phase
of the oscillator at $x\in I$ on $\SS$ at time $t\ge 0$. As we will see below, this probability 
distribution is indeed continuous for $t>0$, provided that the initial conditions 
for  the discrete problem \eqref{KM}, \eqref{KM-ic} converge  weakly  to the probability 
distribution with density \eqref{MF-ic}.
In fact, in \cite{ChiMed16} it is shown that in this case, the empirical measure
on the Borel subsets of $G:=\SS\times I$, $\mathcal{B}(G)$,
\be\lbl{EM2}
\mu_t^n(A)=n^{-1} \sum_{i=1}^n \1_A ((u_{ni}(t), x_{ni})) (A), \quad \; A\in\mathcal{B}(G),
\ee
converges weakly to the absolutely continuous measure
\be\lbl{cont-meas}
\mu_t(A)=\stackrel[A]{}{\int\int} \rho(t,u,x)dudx, \quad \; A\in\mathcal{B}(G).
\ee
The analysis in \cite{ChiMed16}, which extends the analysis of the all-to-all coupled KM
\eqref{classKM} by Lancellotti \cite{Lan05}, relies on the Lipschitz continuity of $W$. 
This is the essential assumption of the Neunzert's fixed point argument that lies at the 
heart of the method used in \cite{Lan05, ChiMed16}. This puts the KM on such common graphs
as the small-world and $k$-nearest-neighbor ones out of the scope of 
applications of \cite{ChiMed16} (see Example~\ref{ex.graphs}). 
It is the goal of the present paper to fill this gap.
Specifically, we extend the Neunzert's method to the KM on convergent families of graphs
with non-Lipschitz limits. Our results apply to a general model of $n$
interacting particles on a graph (cf.~\cite{Gol16}). However, for concreteness and in
view of the diverse applications of the KM,  in this paper,
we present our method in the context of the KM of coupled phase oscillators.

The organization of this paper is as follows. In the next section, we revise the 
Neunzert's fixed point theory to adapt it to the KM on convergent graph
sequences.  This includes a careful choice of the underlying metric space in 
Subsection~\ref{sec.mspace}, setting up the fixed point equation in Subsection~\ref{sec.character},
proving existence and uniqueness of solution of the fixed point equation in
Subsection~\ref{sec.fixed}, and showing continuous dependence on the initial data in Subsection~\ref{sec.ic}
and on the graphon $W$ in Subsection~\ref{sec.kernel}. In Section~\ref{sec.apply},
we apply the fixed point theory to the KM on graphs. To this end, we first 
apply it to an auxiliary problem and then show that this problem approximates
the original KM on graphs. We conclude with a brief discussion of our results
in Section~\ref{sec.discuss}.

\section{The fixed point equation} \lbl{sec.fixed}
\setcounter{equation}{0}

\subsection{The metric space} \lbl{sec.mspace}

Let $\M$ denote  the space of Borel probability measures on
$\SS$. The bounded Lipschitz 
distance on $\M$ is given by 
\be\lbl{BLd}
d(\mu,\eta)=\sup_{f\in \mathcal{L}} \left| \int_\SS f(v) \left( d\mu (v) -
  d\eta (v)\right)\right|, 
\ee
where
$\mathcal{L}$ stands for the class of Lipschitz continuous functions 
on $\SS$ with Lipschitz constant at most $1$ (cf.~\cite{Dud02}).
$\langle \M, d\rangle$ is a complete metric space.

Consider the set of measurable $\M$-valued functions\footnote{
$\bar\mu: I\to \M$ is called measurable if the preimage of an open set in $\M$ is a Lebesgue
measurable subset of $I$.
} 
$\bar\mu: x\mapsto \mu^x\in \M$
$$
\bar \M:=\left\{ \bar\mu: I\to \M \right\}.
$$
Equip $\bar\M$ with the metric
$$
\bar d(\bar \mu,\bar\eta)=\int_I d(\mu^x,\eta^x) dx.
$$

\begin{lem}
$\langle \bar\M, \bar d\rangle$ is a complete metric space.
\end{lem} 
\begin{proof}
Since $d$ is a metric, it is straightforward that $\bar{d}$ is a metric as well. In order to prove the completeness of $\langle \bar\M, \bar d\rangle$, take a Cauchy sequence $\{\bar{\mu}_n\}$ in $\bar{\M}$. Then there is an increasing sequence of indices $n_k$ such that 
$$\bar{d}(\mu_{n_k},\mu_{n_{k+1}})=\int_I d(\mu^x_{n_k},\mu^x_{n_{k+1}})\,dx<\frac{1}{2^{k+1}},\qquad k=1,2,\ldots.$$
By B. Levi's theorem, the series 
$$\sum_{k=1}^\infty d(\mu^x_{n_k},\mu^x_{n_{k+1}})$$
converges for a.e. $x\in I$ to some measurable function $f(x)$, and
$$\sum_{k=1}^\infty\int_I d(\mu^x_{n_k},\mu^x_{n_{k+1}})\,dx=\int_If(x)\,dx.$$
Since, for every $i,j$ with $j>i$,
$$ d (\mu^x_{n_i},\mu^x_{n_j})\le\sum_{k=i}^{j-1} d(\mu^x_{n_k},\mu^x_{n_{k+1}}),$$
the sequence $\{\mu^x_{n_k}\}$ is Cauchy for a.e. $x\in I$. Since the metric space 
$\langle \M, d\rangle$ is complete, there exists the limit 
$$\mu^x=\lim_{k\to\infty}\mu^x_{n_k},\qquad {\rm a.e.}\ x\in I.$$
Extending the definition of $\mu^x$ in an arbitrary way to all of $I$, we obtain a function 
$$\bar{\mu}:=\{\mu^x\} \colon I\to\M$$
which is measurable as an a.e. pointwise limit of measurable functions. Thus 
$\bar{\mu}\in\bar{\M}$. Next, for every $i,j$ with $j>i$, we have
\begin{multline*}
\bar{d}(\bar{\mu}_{n_i},\bar{\mu}_{n_j})=\int_I d (\mu^x_{n_i},\mu^x_{n_j})\,dx
\le\int_I\sum_{k=i}^{j-1} d(\mu^x_{n_k},\mu^x_{n_{k+1}})\,dx
=\sum_{k=i}^{j-1}\int_I d(\mu^x_{n_k},\mu^x_{n_{k+1}})\,dx\\
\le\sum_{k=i}^\infty\int_I d(\mu^x_{n_k},\mu^x_{n_{k+1}})\,dx<\frac{1}{2^i}.
\end{multline*}
We also have that, for all $j>i$, $d(\mu^x_{n_i},\mu^x_{n_j})\le f(x)$ a.e. $x\in I$ 
and the function $f$ is Lebesque integrable on $I$. By the Lebesque Dominated Convergence Theorem, letting $j\to \infty$, we obtain that
$$\bar{d}(\bar{\mu}_{n_i},\bar{\mu})=\int_I d(\mu^x_{n_i},\mu^x)\,dx=
\int_I\lim_{j\to\infty} d(\mu^x_{n_i},\mu^x_{n_j})\,dx=
\lim_{j\to\infty}\int_I d(\mu^x_{n_i},\mu^x_{n_j})\,dx
\le\frac{1}{2^i}\to 0\quad 
{\rm as}\ i\to\infty.$$
Since the subsequence $\{\bar{\mu}_{n_i}\}$ of the Cauchy sequence $\{\bar{\mu}_n\}$ converges to $\bar{\mu}$ in $\bar{\M}$, we conclude that the sequence $\{\bar{\mu}_n\}$ converges to $\bar{\mu}$ as well. 
Thus $\langle\bar{\M},\bar{d}\rangle$ is a complete metric space.
\end{proof}

Let $T>0$ be arbitrary but fixed and denote $\T=[0,T]$. 
We define $\M_\T=C(\T,\bar\M)$, the space of continuous $\bar\M$-valued functions.

For any $\alpha>0,$ the following is a metric on $\bar\M_\T$:
\be\lbl{metr}
 d_\alpha (\bar\mu_.,\bar \nu_.) = 
\sup_{t\in \T} e^{-\alpha t} \bar d(\bar\mu_t,\bar\nu_t)=
\sup_{t\in \T} e^{-\alpha t} \int_I  d(\mu_t^x, \nu_t^x) dx.
\ee

\subsection{The equation of characteristics} \lbl{sec.character}

Recall that $\T:=[0,T]$, where $T>0$ is arbitrary but fixed.
For a given $\bar\mu_.\in\M_\T,$ consider the following equation of 
characteristics
\be\lbl{CE}
{d\over dt} u = V[W, \bar\mu_.](u,x,t), 
\ee
where
\be\lbl{VF}
V[W, \bar\mu_.](u,x,t)=\int_I W(x,y) \left\{\int_\SS D(v-u) d\mu^y_t(v)\right\} dy.
\ee

\begin{lem}\lbl{lem.V}
For every $\bar\mu_.\in \bar\M_\T$, $V[W, \bar\mu_.](u,x,t)$ is Lipschitz continuous
in $u$ and continuous in $x$ and $t$.
\end{lem}
\begin{proof}
The proof follows from the following estimates. First, using \eqref{bounds}
and \eqref{Lip-D}, we have
\be\nonumber
\begin{split}
\left| V[W, \bar \mu_.](u,x,t)-V[W, \bar\mu_.](v,x,t) \right| &=
\left|\int_I W(x,y) \int_\SS \left(D(w-u)-D(w-v)\right) d\mu^y_t(w) dy\right|\\
& \le 
\int_I \left| W(x,y) \right|\int_\SS \left| D(w-u)-D(w-v)\right| d\mu^y_t(w) dy\\
& \le \left|u-v\right|,\quad \forall u,v\in \SS,\; x\in I, \; t\in \T.
\end{split}
\ee
Using the bound on $D$ (cf.~\eqref{bounds}), we obtain
\be\lbl{V-of-x}
\begin{split}
\left| V[W,\bar \mu_.](u,x,t)-V[W,\bar\mu_.](u,z,t) \right| &=
\left|\int_I \left(W(x,y)-W(z,y)\right) \int_\SS D(v-u)d\mu_t^y(v)dy\right|\\
& \le \int_I \left|W(x,y)-W(z,y) \right| \int_\SS \left|D(v-u)\right|d\mu_t^y(v)dy\\
&\le  \int_I \left|W(x,y)-W(z,y) \right|dy \quad  \forall u\in\SS,\, x,z\in I,\; t\in\T.
\end{split}
\ee
The continuity of $V[W,\bar\mu_.]$ in $x$ follows from \eqref{V-of-x} and
\eqref{Wcont}. 

Finally,  
\begin{equation*}
\begin{split}
\left| V[W,\bar \mu_.](u,x,t)-V[W,\bar\mu_.](u,x,s) \right| & =
\left|\int_I W(x,y) \int_\SS D(v-u)\left( d\mu_t^y(v) - d\mu_s^y(v)\right) dy\right|\\
&\le \int_I \left| W(x,y) \right| \left| \int_\SS  
 D(v-u)\left( d\mu_t^y(v) - d\mu_s^y(v)\right) \right| dy\\
&\le \int_I \left| W(x,y) \right| d(\mu_t^y, \mu_s^y)dy \\
&\le \bar d(\bar \mu_t,\bar\mu_s).
\end{split}
\end{equation*}
\end{proof}

Similarly to the derivation of the last inequality, we prove the following lemma. 
\begin{lem}\lbl{lem.Lip-mu}
\be\lbl{V-Lip-mu}
\left| V[W,\bar\mu_.](u,x,t) -  V[W,\bar\nu_.](u,x,t)\right|\le \bar{d} (\bar\mu_t,\bar\nu_t)
\quad \forall \bar\mu_.,\bar\nu_. \in \bar\M_\T.
\ee
\end{lem}
% \begin{proof}
% \begin{equation*}
% \begin{split}
% \left| V[W,\bar\mu_.](u,x,t) -  V[W,\bar\eta_.](u,x,t)\right| 
% &\le \left|\int_I W(x,y) \int_\SS D(v-u) \left[ d\mu_t^y(v)- d\eta_t^y(v)\right] dy\right| \\
% &\le \int_I \left|\int_\SS D(v-u) \left[ d\mu_t^y(v)- d\eta_t^y(v)\right]\right| dy\\
% &\le \int d(\mu_t^y,\eta_t^y) dy = \bar d(\bar \mu_t, \bar\nu_t).
% \end{split}
% \end{equation*}
% \end{proof}

Consider the initial value problem (IVP) for
\eqref{CE} subject to the initial condition at time $s\in\T,$ $u(s)=u_s$. 
By Lemma~\ref{lem.V},  for every $x\in I$ and $u_s\in\SS,$ there exists 
a unique solution of the IVP for \eqref{CE}.
% Denote this
% solution by $u(t;u_0)$. 
Since $V[W, \bar\mu_.](u,x,t)$ is uniformly Lipschitz
in $u$, $u(t)$ can be extended to $t\in\T$. 
% It depends 
% continuously on $x\in I$ as a parameter. 
Thus, the equation of characteristics \eqref{CE}
generates the flow on $\SS:$
\be\lbl{flow}
T^x_{t,s}[W, \bar\mu_.] u_s = u(t), \; t,s\in \T, u_s\in\SS. 
\ee
For every $x\in I$, $T^x_{t,s}[W, \bar\mu_.], t,s\in \T,$ is a two-parameter family of one-to-one
transformations of $\SS$ to itself depending continuously on $x$:
$$
T^x_{s,s}[W, \bar\mu_.]=\operatorname{id},\quad
({T^x_{t,s}}[W, \bar\mu_.])^{-1}=T^x_{s,t}[W, \bar\mu_.].
$$

\subsection{Existence of solution of the fixed point equation} \lbl{sec.exist}

In the remainder of this section, we will study the following fixed point equation.
For a given $\bar\mu_0\in\bar\M$, consider the pushforward measure
\be\lbl{FP}
\bar\mu_t =\bar\mu_0\circ T_{0,t} [W, \bar\mu_.] \quad \forall t\in \T,
\ee
which is interpreted as
\be\lbl{push}
\mu^x_t=\mu_0^x \circ T^x_{0,t} [W, \bar\mu_.] \quad \mbox{almost everywhere (a.e.)}\; x\in I,\;
\mbox{and}\; t\in \T.
\ee

First, we address existence and uniqueness of solution of \eqref{FP}.  
\begin{thm}\lbl{thm.exist}
For every $\bar\mu_0 \in\bar\M$, the fixed point equation 
\eqref{FP} has 
a unique solution $\bar\mu_.\in \bar\M_\T$. 
\end{thm}

For the proof of Theorem~\ref{thm.exist},  we will need a variant of the Gronwall's 
inequality, which we formulate below for convenience.
   
\begin{lem}\lbl{lem.Gron}
Let $\phi(t)$ and $a(t)$ be  continuous functions on $[0,T]$ and
\be\lbl{Gron-cond}
\phi(t)\le A\int_0^t\phi(s)ds + B \int_0^t a(s) ds +C,\quad t\in [0,T],
\ee
where $A\ge 0$. 
Then
\be\lbl{Gron-conclusion}
\phi(t)\le  e^{At}\left( B\int_0^t a(s) e^{-As}ds +C\right).
\ee
\end{lem}
\begin{proof}
Denote the right-hand side of \eqref{Gron-cond} by $\psi(t)$.
Differentiating $\psi(t)$ and using \eqref{Gron-cond}, we have
\be\lbl{Gron-step-1}
\psi^\prime(s)=A \phi(s) + B a(s)\le A\psi(s) +Ba(s).
\ee
Thus,
$$
{d\over ds} \left\{e^{-As} \psi(s)\right\} \le  B e^{-As} a(s)
$$
and 
\be\lbl{Gron-step-2}
\begin{split}
\psi(t) &\le e^{At} \left(\psi(0) +B  \int_0^t e^{-As} a(s)ds \right) \\
           & \le e^{At} \left(B \int_0^t e^{-As} a(s)ds +C\right).
\end{split}
\ee
Finally,
$$
\phi(t)\le \psi(t)\le e^{At} \left(B \int_0^t e^{-As} a(s)ds +C\right).
$$
\end{proof}

\begin{proof}[Proof of Theorem~\ref{thm.exist}]
Given $\bar{\mu}_0\in\bar\M,$ consider $A:\bar\M_\T\to\bar\M_\T$ defined by
\be\lbl{def-A}
A[W,\bar\mu_.](t,x)=\mu_0^x\circ T^x_{0,t} [W,\bar\mu_.],\quad\mbox{a.e.}\; x\in I.
\ee
Below we show that $A$ is a contraction on $(\bar\M_\T, d_\alpha)$ with 
$\alpha>2$.
To this end, 
\be\lbl{do-1}
\begin{split}
\bar d\left( A[W,\bar\mu_.](t,\cdot), \right.&\left. A[W,\bar\eta_.](t,\cdot)\right) =
\bar d(\bar\mu_0\circ T_{0,t}[W,\bar\mu_.],\bar\mu_0\circ T_{0,t}[W,\bar\eta_.])\\ 
&=\int_I d(\mu_0^x\circ T^x_{0,t} [W,\bar\mu_.], 
\mu_0^x\circ T^x_{0,t} [W,\bar\eta_.])dx\\
&=\int_I \sup_{f\in\mathcal{L}} \left| \int_\SS f(v) \left( d\mu_0^x\circ 
T_{0,t}^x[W,\bar\mu_.] (v)- d\mu_0^x\circ T_{0,t}^x[W,\bar\eta_.] (v)\right) \right|dx\\
&=
\int_I \sup_{f\in\mathcal{L}} 
\left| \int_\SS f(T^x_{t,0}[W,\bar\mu_.]v) d\mu_0^x(v) -
\int_\SS f(T^x_{t,0}[W,\bar\eta_.]v) d\mu_0^x(v)
\right|dx\\
&\le 
\int_I\int_\SS \left| T^x_{t,0}[W,\bar\mu_.]v -T^x_{t,0}[W,\bar\eta_.]v\right|
d\mu_0^x(v) dx=:\lambda(t).
\end{split}
\ee
The change of variables formula used in \eqref{do-1} is explained in
\cite[\S 6.1]{MakPod13}.
Using \eqref{CE} and \eqref{V-Lip-mu}, we obtain
\be\lbl{lambda}
\begin{split}
\lambda(t) &=\int_I\int_\SS \left| T^x_{t,0}[W,\bar\mu_.]v -T^x_{t,0}[W,\bar\eta_.]v\right|
d\mu_0^x(v) dx\\
&\le 
\int_0^t\int_I\int_\SS \left| V[W,\bar\mu_.]\left(T_{s,0}^x[W,\bar\mu_.]v,x,s)\right)-
V[W,\bar\eta_.]\left(T_{s,0}^x[W,\bar\eta_.]v,x,s)\right)\right|d\mu_0^x(v) dxds\\
& \le 
\int_0^t\int_I\int_\SS \left| V[W,\bar\mu_.]\left(T_{s,0}^x[W,\bar\mu_.]v,x,s)\right)-
V[W,\bar\eta_.]\left(T_{s,0}^x[W,\bar\mu_.]v,x,s)\right)\right|d\mu_0^x(v) dxds\\
&+
\int_0^t\int_I\int_\SS \left| V[W,\bar\eta_.]\left(T_{s,0}^x[W,\bar\mu_.]v,x,s)\right)-
V[W,\bar\eta_.]\left(T_{s,0}^x[W,\bar\eta_.]v,x,s)\right)\right|d\mu_0^x(v) dx ds \\
&\le
\int_0^t \bar d(\bar\mu_s,\bar\eta_s)ds + 
\int_0^t\int_I \int_\SS \left| T_{s,0}^x[W,\bar\mu_.]v -T_{s,0}^x[W,\bar\eta_.]v\right|
d\mu_0^x(v) dx ds\\
&\le \int_0^t  \bar d(\bar\mu_s,\bar\eta_s)ds +
\int_0^t  \lambda(s)ds.
\end{split}
\ee
Using Gronwall's inequality (cf. Lemma~\ref{lem.Gron}), 
from \eqref{lambda} we obtain
\be\lbl{2nd-lambda}
\lambda(t) \le e^t\int_0^t \bar d(\bar\mu_s,\bar\eta_s)e^{-s}ds.
\ee

Combining \eqref{do-1}, \eqref{lambda}, and \eqref{2nd-lambda},
we have
\be
\bar d\left( A[W,\bar\mu](t,\cdot),A[W,\bar\eta](t,\cdot)\right) \le
 e^t\int_0^t \bar d(\bar\mu_s,\bar\eta_s)e^{-s}ds.
\ee
and
\be\lbl{we-have}
\begin{split}
d_\alpha \left( A[W,\bar\mu](t,\cdot), A[W,\bar\eta](t,\cdot)\right) & =
\sup_{t\in\T} 
\left\{ e^{-\alpha t} \bar d\left( A[W,\bar\mu](t,\cdot),A[W,\bar\eta](t,\cdot)\right)
\right\} \\
& \le \sup_{t\in\T} 
e^{-(\alpha-1)t}\int_0^t \bar d(\bar\mu_s,\bar\eta_s)e^{-s}ds\\
&\le d_\alpha \left( \bar\mu,  \bar\eta\right) e^{-(\alpha-1)t}\int_0^t e^{(\alpha-1) s}ds\\
& \le (\alpha -1)^{-1}  d_\alpha \left( \bar\mu_.,  \bar\eta_.\right).
\end{split}
\ee
We conclude the proof with using the contraction mapping principle to establish
a unique solution of \eqref{FP}.
\end{proof}

\subsection{Continuous dependence on initial data} \lbl{sec.ic}
\begin{lem}\lbl{lem.initial-data}
Let $\bar\mu_.,\bar\eta_.\in\bar\M_\T$ be two solutions of \eqref{FP} corresponding to
initial conditions $\bar\mu_0, \bar\eta_0\in \bar\M$ respectively. Then
\be\lbl{initial-data}
\sup_{t\in\T}\bar d(\bar\mu_t,\bar\eta_t)\le  e^T \bar{d}(\bar\mu_0,\bar\eta_0).
\ee
\end{lem}
\begin{proof}
For every $t\in\T$, by the triangle inequality, we have
\be\lbl{1st-triangle}
\begin{split}
\bar d(\bar\mu_t,\bar\eta_t) &= 
\bar d(\bar\mu_0\circ T_{0,t}[W,\bar\mu_.],\bar\eta_0\circ T_{0,t}[W,\bar\eta_.]) \\
& \le \bar d(\bar\mu_0\circ T_{0,t}[W,\bar\mu_.],\bar\mu_0\circ T_{0,t}[W,\bar\eta_.])
+\bar d(\bar\mu_0\circ T_{0,t}[W,\bar\eta_.],\bar\eta_0\circ T_{0,t}[W,\bar\eta_.]).
\end{split}
\ee
Exactly in the same way as in \eqref{do-1}, we estimate the first term on the right hand 
side of \eqref{1st-triangle} as follows
\be\lbl{1st-term}
\begin{split}
\bar d(\bar\mu_0\circ T_{0,t}[W,\bar\mu_.],\bar\mu_0\circ T_{0,t}[W,\bar\eta_.]) 
&=
\int_I \sup_{f\in\mathcal{L}} 
\left| \int_\SS \left(f(T^x_{t,0}[W,\bar\mu_.]v) - f(T^x_{t,0}[W,\bar\eta_.]v) \right) d\mu_0^x(v)
\right| dx\\
&\le 
\int_I\int_\SS \left| T^x_{t,0}[W,\bar\mu_.]v -T^x_{t,0}[W,\bar\eta_.]v\right|
d\mu_0^x(v) dx=:\lambda(t).
\end{split}
\ee
Similarly, repeating the steps in \eqref{lambda} 
\be\lbl{lambda-1}
\lambda(t)\le \int_0^t  \bar d(\bar\mu_s,\bar\eta_s)ds +
\int_0^t  \lambda(s)ds.
\ee
Using Gronwall's inequality, from \eqref{lambda-1} we obtain
\be\lbl{2nd-lambda-a}
\lambda(t) \le e^t\int_0^t \bar d(\bar\mu_s,\bar\eta_s)e^{-s}ds.
\ee
Next, we turn to the second term on the right--hand side of \eqref{1st-triangle}:
\be\lbl{2nd-term}
\begin{split}
\bar d(\bar\mu_0\circ T_{0,t}[W,\bar\eta_.],\bar\eta_0\circ T_{0,t}[W,\bar\eta_.])
&=
\int_I \sup_{f\in\mathcal{L}}  \left| \int_\SS f(v) d\mu_0^x\circ T^x_{0,t}[W,\bar\eta_.](v)
-\int_\SS f(v) d\eta_0^x\circ T^x_{0,t}[W,\bar\eta_.](v)\right| dx\\
&=
\int_I \sup_{f\in\mathcal{L}}  \left| \int_\SS f(T^x_{t,0}[W,\bar\eta_.] v) d\mu_0^x(v)
-\int_\SS f(T^x_{t,0}[W,\bar\eta_.] v) d\eta_0^x(v)\right| dx\\
&\le  \bar d(\bar\mu_0,\bar\eta_0).
\end{split}
\ee
The combination of the \eqref{1st-triangle}, \eqref{1st-term}, 
\eqref{2nd-lambda-a}, and \eqref{2nd-term} yields
\be\lbl{pre-Gron}
\bar d(\bar\mu_t,\bar\eta_t)\le 
e^t\int_0^t \bar d(\bar\mu_s,\bar\eta_s) e^{-s}ds+
\bar d(\bar\mu_0,\bar\eta_0).
\ee
Denote $\phi(t):= \bar d(\bar\mu_t,\bar\eta_t)e^{-t}$ and rewrite \eqref{pre-Gron} as
$$
\phi(t)\le \int_0^t \phi(s)ds + e^{-t}\bar d(\bar\mu_0,\bar\eta_0)=:\psi(t).
$$
Arguing as in the proof of Lemma~\ref{lem.Gron}, we derive
\begin{equation*}
\begin{split}
\psi^\prime(t) &= \phi(t) - e^{-t} \bar d(\bar\mu_0,\bar\eta_0) \\
 &\le \psi(t) - e^{-t} \bar d(\bar\mu_0,\bar\eta_0)\\
&\le \psi(t).
\end{split}
\end{equation*}
Further,
\be\lbl{dif-ineq}
 {d\over ds} \left\{e^{-s}\psi(s)\right\}= e^{-s}\psi^\prime(s) -e^{-s}\psi(s) \le 0.
\ee
and, thus,
$$
\phi(t)\le \psi(t)\le e^t \psi(0)= \bar d(\bar\mu_0,\bar\eta_0).
$$
Recalling, the definition of $\phi(t)$, we arrive at
\be\lbl{cont-dep}
\bar d(\bar\mu_t,\bar\eta_t)\le e^t \bar d(\bar\mu_0,\bar\eta_0),\quad t\in\T,
\ee
from which \eqref{initial-data} follows.
\end{proof}

\subsection{Continuous dependence on the kernel} \lbl{sec.kernel}
In this subsection, we study how the solution of the fixed point equation \eqref{FP}
changes under the perturbation of the kernel $W$.
To this end, let $W$ and $U$ be two bounded measurable functions on $I^2$
satisfying \eqref{Wcont}. Then for a given $\mu_0\in\M$ each of the fixed point equations
\begin{eqnarray}
\lbl{FPa}
\bar\mu_t &=& \bar\mu_0\circ T_{0,t}[W, \bar\mu_.]\\
\lbl{FPb}
\bar\nu_t &=& \bar\mu_0\circ T_{0,t}[U, \bar\nu_.].
\end{eqnarray}
has a unique solution in $\M_\T$, which we denote by $\bar\mu_t$ and $\bar\nu_t$ respectively. 
\begin{lem}\lbl{lem.c-kernel}
\be\lbl{c-kernel}
\sup_{t\in\T}\bar d(\bar\mu_t, \bar\nu_t) \le e^{2T} \|W-U\|_{L^1(I^2)}.
\ee
\end{lem}
\begin{proof}
We proceed along the lines of the proof of Lemma~\ref{lem.initial-data}.
Replicating the steps in \eqref{do-1}, we derive
\be\lbl{do-1-again}
\bar d(\bar\mu_t,\bar\nu_t) =\int_I\int_\SS \left| T^x_{t,0}[W,\bar\mu_.]v-
 T^x_{t,0}[U,\bar\nu_.]v\right|d\mu^x_0 dx =:\lambda(t).
\ee
Further,
\be\lbl{lambda-again}
\begin{split}
\lambda(t) &=\int_I\int_\SS \left|  \int_0^t  \left\{V[W,\bar\mu.] \left(T^x_{t,0}[W,\bar\mu_.]v, x,s\right)-
 V[U,\bar\nu_.]\left(T^x_{t,0}[U,\bar\nu_.]v, x,s\right) \right\} ds \right|d\mu^x_0(v) dx\\
&\le \int_0^t \int_I\int_\SS
\left| V[W,\bar\mu.] \left(T^x_{t,0}[W,\bar\mu_.]v, x,s\right)-
 V[W,\bar\nu_.]\left(T^x_{t,0}[W,\bar\mu_.]v, x,s\right) \right|d\mu^x_0(v) dxds \\
&+\int_0^t \int_I\int_\SS
\left| V[W,\bar\nu.] \left(T^x_{t,0}[W,\bar\mu_.]v, x,s\right)-
 V[U,\bar\nu_.]\left(T^x_{t,0}[W,\bar\mu_.]v, x,s\right) \right|d\mu^x_0(v) dxds \\
&+\int_0^t \int_I\int_\SS
\left| V[U,\bar\nu.] \left(T^x_{t,0}[W,\bar\mu_.]v, x,s\right)-
 V[U,\bar\nu_.]\left(T^x_{t,0}[U,\bar\nu_.]v, x,s\right) \right|d\mu^x_0(v) dxds\\
& =:\lambda_1(t)+\lambda_2(t)+\lambda_3(t).
\end{split}
\ee
We estimate the first term on the right--hand side of \eqref{lambda-again}, using
Lemma~\ref{lem.Lip-mu}:
\be\lbl{1st-lambda}
\lambda_1(t)\le \int_0^t \bar d(\bar\mu_s,\bar\nu_s) ds.
\ee
For the second term, we have
\be\lbl{2nd-lambda-b}
\begin{split}
\lambda_2(t) &\le \int_0^t \int_I\int_\SS
\left[ \int_I \left| W(x,y)-U(x,y)\right| 
\left\{ \int_\SS \left| D(w- T^x_{t,0}[W, \bar\mu_.)v) \right| d\nu^y_s(w)\right\} dy
% \left\{ \int_\SS \left| D(v-u) \right| d\nu^y_s(v)\right\} dy
\right]
d\mu^x_0(v) dxds \\
& \le \int_0^t \int_{I^2} |W(x,y)-U(x,y)| dxdyds,
\end{split}
\ee
where we used $|D(u)|\le 1$ to get the latter inequality.
Finally, to estimate the third term, we use Lipschitz continuity of 
$V[\cdot,\cdot](u,x,t)$ in $u$:
\be\lbl{3rd-lambda}
\lambda_3(t) \le \int_0^t
\int_I\int_\SS \left| T^x_{s,0}[W,\bar\mu_.]v-
 T^x_{s,0}[U,\bar\nu_.]v\right|d\mu^x_0(v) dx ds=\int_0^t\lambda(s)ds.
\ee

Plugging \eqref{1st-lambda}-\eqref{3rd-lambda} into 
\eqref{lambda-again}, we obtain
\be\lbl{ineq-lambda}
\lambda(t)\le \int_0^t \lambda(s)ds +
\int_0^t \left(\bar d(\bar\mu_s,\bar\nu_s) +\|W-U\|_{L^1(I^2)} \right) ds.
\ee
By Gronwall's inequality (cf. Lemma~\ref{lem.Gron}),
\be\lbl{bound-lambda}
\lambda(t)\le e^t\int_0^t e^{-s} 
\left(\bar d(\bar\mu_s,\bar\nu_s) +\|W-U\|_{L^1(I^2)} \right) ds.
\ee
The combination of \eqref{do-1-again} and \eqref{bound-lambda}
yields
\be\lbl{almost-done}
\bar d(\bar\mu_t,\bar\nu_t)\le  e^t\int_0^t e^{-s} 
\left(\bar d(\bar\mu_s,\bar\nu_s) +\|W-U\|_{L^1(I^2)} \right) ds.
\ee
Denote $\phi(t):= e^{-t} \bar d(\bar\mu_t,\bar\nu_t)$ and rewrite
\eqref{almost-done} as
\be\lbl{rewrite-almost-done}
\begin{split}
\phi(t) & \le \int_0^t\phi(s) ds +  \|W-U\|_{L^1(I^2)} \int_0^t e^{-s} ds\\
&\le \int_0^t\phi(s) ds +  \|W-U\|_{L^1(I^2)}.
\end{split}
\ee
Using Lemma~\ref{lem.Gron} again, we have
$$
\phi(t) \le e^t \|W-U\|_{L^1(I^2)}.
$$
Recalling, the definition of $\phi(t)$, we finally get
$$
\bar d(\bar\mu_t,\bar\nu_t)\le  e^{2t} \|W-U\|_{L^1(I^2)}.
$$
\end{proof}

\section{Application to coupled systems}\lbl{sec.apply}
\setcounter{equation}{0}
In this section, we apply the fixed point theory developed in the previous section
to the proof of convergence of solutions of the KM on graphs \eqref{KMonG}.

\subsection{The initial value problem}\lbl{sec.ivp}
\setcounter{equation}{0}
We begin by addressing the wellposedness of the IVP for \eqref{MF}, i.e., review the 
notion of the weak solution of the mean field equation \eqref{MF} 
and then prove the existence and uniqueness of the weak solution of the IVP 
\eqref{MF}-\eqref{MF-ic}. The following definition of the weak solution of \eqref{MF} is adapted
from \cite{Neu78}. 

\begin{df}
A measurable function $\rho: \T\times G\to \R$ is called a weak
solution of 
\eqref{MF}-\eqref{MF-ic}
if the following conditions hold a.e. $x\in I$.
\begin{enumerate} 
\item
$\rho(t,u,x)$ is weakly continuous in $t\in\T$, i.e., 
$t\mapsto \int_\SS \rho(t,u,x)f(u) du$ is a continuous map for every
$f\in C(\SS)$.
\item 
The following identity holds
\be\lbl{weak-MF}
\int_0^T \left\{ \int_\SS \rho(t,u,x) \left( {\p\over \p t} w(t,u) +V(t,u,x) {\p \over \p u} w(t,u) \right)du
\right\}dt +\int_\SS w(0,u) \rho_0(u,x) du=0
\ee
for every $w\in C^1(\T\times \SS)$ with support in $(0,T]\times \SS$.
\end{enumerate}
\end{df}

\begin{thm}\lbl{thm.wellp}
Suppose $W\in L^1(I^2)$ satisfies \eqref{Wcont} and $\rho_0\in L^1(\SS\times I)$.
Then there is a unique weak solution to the IVP \eqref{MF}-\eqref{MF-ic}.
\end{thm}
\begin{proof}
Recall that $T_{0,t}^x:=T_{0,t}^x[W,\bar\mu_.]$ is the flow generated by the equation of 
characteristics \eqref{CE} on $\SS$ (see \eqref{flow}). 
By Theorem~\ref{thm.exist}, for $\mu^x_0$, a family of measures on $(\SS,\mathcal{B}(\SS))$
with densities $\rho_0(\cdot,x)$, there is a unique solution of the fixed point equation 
\eqref{FP}.

For a.e. $x\in I$ and
$t\in \T$, $T_{0,t}^x$ is one-to-one and Lipschitz continuous. As an absolutely continuous
function, $T_{0,t}^x$ is differentiable a.e. on $\SS$, and has an essentially bounded weak
derivative. 
% For $x\in I$ and $t\in\T$, let $O^x_t\subset \SS$ be an open set of full measure such that
% $T_{0,t}^x$ is an injective and continuously differentiable on $O^x_t$. 
Using \eqref{push} and the change of variables formula
for Lipschitz maps (cf.~\cite[Theorem~2, \S 3.3.3]{EvansGariepy}), we 
have
\be\lbl{change-var}
\begin{split}
\mu_t^x(A)& =\mu_0^x \circ T_{0,t}^x (A)=  \int_{T_{0,t}^x A} \rho_0(u,x) du\\
% &= \int_{T_{0,t}^x (A\cap O_t^x)}\rho_0(u,x) du = 
% \int_{A\cap O_t^x}\rho_0( T_{0,t}^x u,x) \left|{\p \over \p u} T_{0,t}^x u\right| du\\
&
=\int_{A}\rho_0( T_{0,t}^xu ,x) \left|{\p \over \p u} T_{0,t}^x u\right| du, \quad \mbox{a.e.} \;x\in I,
\end{split}
\ee
for any Borel $A\subset \SS$.
In the last integral of \eqref{change-var}, ${\p \over \p u} T_{0,t}^x u$ is understood as a 
weak derivative.
Thus, for $t\in \T$ and a.e. $x\in I$,  $\mu_t^x$ is an absolutely continuous measure with density
\be\lbl{t-density}
\rho(t,u,x) = \rho_0( T_{0,t}^xu ,x) \left|{\p \over \p u} T_{0,t}^x u\right|.
\ee

To show that $\rho(t,u,x)$ is a weak solution of \eqref{MF}-\eqref{MF-ic}, as in the proof of Thorem~1 
in \cite{Neu78}, we 
set
\be\lbl{def-h}
h:={\p\over \p t} w+V(t,u,x) {\p \over \p u} w(t,u)
\ee
and compute
\begin{equation}\lbl{check-weak}
\begin{split}
\int_0^T \int_\SS \rho(t,u,x) h(t,u,x) du dt &  =
\int_0^T \int_\SS \left(\rho_0( T_{0,t}^xu ,x) 
\left|{\p \over \p u} T_{0,t}^x u\right| h(t,u,x) \right) dudt\\
&=\int_\SS \rho_0 (u,x) \left( \int_0^T  h(t, T_{t,0}^x u) dt\right) du.
\end{split}
\end{equation}
Further, using the chain rule and \eqref{def-h}, we have
\be\lbl{chain-rule}
\begin{split}
{d\over d t} w(t, T_{t,0}^x u)& = \p_1  w(t, T_{t,0}^x u) +
\p_2 w(t, T_{t,0}^x u) {d\over dt} T_{t,0}^x u \\
&= \p_1  w(t, T_{t,0}^x u) + \p_2 w(t, T_{t,0}^x u) V(t,u,x)= h(t, T_{t,0}^x u,x),
\end{split}
\ee
where $\p_{1,2}$ stand for the partial derivatives with respect to the first and
second argument respectively.
The combination of \eqref{check-weak} and  \eqref{chain-rule} yields
\begin{equation*}
\begin{split}
\int_0^T \int_\SS \rho(t,u,x) h(t,u,x) du dt & = 
\int_\SS\rho_0 (u,x) \left( \int_0^T  h(t, T_{t,0}^x u) dt\right) du\\
& =\int_\SS \rho_0 (u,x) \left( \int_0^T {\p\over \p t} w(t, T_{t,0}^x u)  du\right) dt\\
&=-\int_\SS \rho_0 (u,x) w(0,u)  du.
\end{split}
\end{equation*}
\end{proof}

\subsection{Approximation}
We continue by collecting several results on approximation,
which will be used later in this section.

Let $W\in L^2(I^2)$ and consider a step function $\bar W_n:I\to\R$, which on 
each cell $I_{n,i}\times I_{n,j}$, $i,j\in [n]$, is equal to the average value
$$
 \bar W_{n,ij}=n^2\int_{I_{n,i}\times I_{n,j}} W(x,y) dxdy.
$$

\begin{lem}\lbl{lem.mart}
$\bar W_n\to W$ a.e. and in $L^2(I^2)$ as $n\to\infty.$
\end{lem}  
\begin{proof}
Consider a probability space $\left( I^2, \mathcal{B}(I^2), \lambda\right)$ with 
sample space $I^2$, Borel $\sigma$--algebra of subsets of $I^2$, and 
the Lebesgue measure as the probability $\lambda$. On this probability space,
$(\bar W_n)$ forms a sequence of $L^2$-bounded martingales \cite{Williams-Prob-Mart}.
Indeed, let $\mathcal{A}_n$ be a $\sigma$--algebra of subsets of $I^2$
generated by $I_{n,i}\times I_{n,j}$, $i,j\in [n]$.
Then $\bar W_n$ is the conditional
expectation
$$
\bar W_n=\E (W|\mathcal{A}_n), \quad n\in\N.
$$
Since $W\in L^2(I^2)$, the statement of the lemma follows
from 
the $L^p$--Martingale Convergence Theorem 
\cite[Theorem 12.1]{Williams-Prob-Mart}.
\end{proof}

Let $\Gamma_n=\langle [n], E(\Gamma_n), W_n\rangle$ be a 
weighted graph on $n$ nodes.
Here, $[n]$ is the set of vertices, 
$$
E(\Gamma_n)=\left\{ \{i,j\}: \; i,j\in [n]\right\} 
$$
is the edge set. Each edge $\{i,j\}$ is equipped with a real weight $W_{n,ij}$, the $ij$th entry
of the $n\times n$ weight matrix $W_n$.

On $\Gamma_n$ consider a coupled system:
\be\lbl{generic}
\dot u_{n,i}  =n^{-1} \sum_{j=1}^n W_{n,ij} D(u_{n,j}-u_{n,i}),\quad i\in [n].
\ee
By $u_n=(u_{n,1}, u_{n,2},\dots, u_{n,n})^\t\in \R^n$ we denote a solution of the coupled system
\eqref{generic}. Along with \eqref{generic}, consider the coupled system on a weighted
graph $\tilde\Gamma_n=\langle [n], E(\Gamma_n), \tilde W_n\rangle$, whose solution is denoted
by $\tilde u_n$.

Recall the discrete $L^2$-norm defined in \eqref{norm-1}. We will also need
its analog for discretizations of functions on $I^2$:
\be\lbl{norms}
\|W_n\|_{2,n}=\sqrt{n^{-2}\sum_{i,j=1}^n W_{n,ij}^2}.
\ee

\begin{lem}\lbl{lem.generic} \cite[Lemma~4.1]{ChiMed16}
Let $u_n(t)$ and $\tilde u_n(t)$ denote solutions of the IVPs for the 
coupled system \eqref{generic} on weighted graphs 
$\Gamma_n=\langle [n], E(\Gamma_n), W_n\rangle$ and
$\tilde\Gamma_n=\langle [n], E(\Gamma_n), \tilde W_n\rangle$ respectively. Suppose that
the initial data for these problems coincide
\be\lbl{bdd-ic} 
u_n(0)=\tilde u_n(0).
\ee 
Then for any $T>0$,  
\be\lbl{ave} 
\max_{t\in [0,T]} \left\|u_n(t)-\tilde u_n(t)\right\|_{1,n} \le C_1 
\left\|W_n-\tilde W_n\right\|_{2,n}, 
\ee
where $C_1=\sqrt{Te^{5T}}>0.$
\end{lem}
\begin{rem}\lbl{rem.equiv} Note that  \eqref{ave} implies that the solutions 
of any two discrete models \eqref{wKM} with weights given by $W_n$ and $\tilde W_n$ 
will be close, provided that $W_n$ and $\tilde W_n$ converge to $W$ in $L^2$.
\end{rem}

\subsection{The auxiliary problem}
Our next goal is to establish convergence to the mean field limit for an auxiliary discrete model.
To this end, recall that
 $W_n:I^2\to [-1,1]$ is a step function taking constant value 
$$
W_{n,ij} :=n^2\int_{I_{n,i}\times I_{n,j}} W(x,y) dxdy
$$ 
on  each cell $I_{n,i}\times I_{n,j},$ $i,j \in [n],$ (cf.~\eqref{weights}).

Let $N=nm$ for some  $m\in\N$ and consider a coupled system
\begin{eqnarray}\lbl{KMn}
\dot v_{N,(k-1)m+l}& = &N^{-1} \sum_{i=1}^n\sum_{j=1}^m W_{n,ki} D(v_{N,(i-1)m+j}-v_{N,(k-1)m+l}),\\
\lbl{KMn-ic}
v_{N,(k-1)m+l}(0)&=& v_{N,(k-1)m+l}^0, \quad k\in [n], \; l\in [m].
\end{eqnarray}
We are going to describe the solution of \eqref{KMn}, \eqref{KMn-ic} in terms of the 
empirical measure
\be\lbl{EMn}
\mu_{n,m,t}^x (A)= m^{-1} \sum_{j=1}^m \1_A (v_{N,(i-1)m+j}(t)), \quad x\in I_{n,i},
\; A\in\mathcal{B}(\SS).
\ee

We will show that in the large $m$ limit, the behavior of solutions of the discrete model
\eqref{KMn}, \eqref{KMn-ic}, is effectively approximated by the IVP for the following
integro-differential equation:
\be\lbl{MFn}
{\p\over \p t} \rho_n(t,u,x) + {\p\over \p u} \left\{ V_n(t,u,x) \rho_n(t,u,x)\right\} =0,
\ee
where
\be\lbl{VFn}
V_n(t,u,x)= \int_I\int_\SS W_n(x,y) D(u-v)\rho_n(t,v,y)dvdy,
\ee
and subject to the initial condition
\be\lbl{MFn-ic}
\rho_n(0,u,x)=\rho^0(u,x).
\ee
About the initial condition $\rho^0(u,x)$ we assume the following. 
% is a nonnegative integrable function in $u$ and is 
% a piecewise continuous function in $x$. 
For each $x\in I$, $\rho^0(\cdot,x)$ is a probability density function, i.e., 
$$
\int_\SS \rho^0(u,x) du =1\quad \forall x\in I.
$$
Furthermore, $x\mapsto \int_\SS f(u)\rho^0(u,x) du$ is Riemann integrable for every 
$f\in C(\SS)$.

By Theorem~\ref{thm.wellp}, there is a unique weak solution of the IVP 
\eqref{MFn}, \eqref{MFn-ic}, which  defines a family  of 
absolutely continuous measures on $(\SS, \mathcal{B}(\SS))$: 
\be\lbl{cmeas-n}
\mu^x_{n,t}(A)=\int_A \rho_n (t,u,x) du,\quad x\in I, \; A\in\mathcal{B}(\SS).
\ee

\begin{thm}\lbl{thm.aux-converge}
% There is a unique weak solution of the IVP \eqref{MFn}, \eqref{MFn-ic}.
% Thus, the family of absolutely continuous measures \eqref{cmeas-n} is well-defined.
Let $n\in\N$ be arbitrary but fixed
and suppose that the the initial values $v_{N,j}^0, j\in [N],$ are independent random variables,
whose distributions have densities 
$\rho^0(\cdot, jN^{-1})$, $j\in [N]$.
Then
\be\lbl{aux-converge}
\lim_{m\to\infty} \sup_{t\in\T} \bar d(\bar\mu_{n,t}, \bar\mu_{n,m,t})=0 \quad \mbox{a.s.}.
\ee
\end{thm}

The proof of Theorem~\ref{thm.aux-converge} follows from two lemmas, which we 
prove next.

\begin{lem} 
% The IVP \eqref{MFn}, \eqref{MFn-ic} has a unique weak solution. 
The empirical measure $\mu_{n,m,t}^x$ and the absolutely continuous measure
$\mu_{n,t}^x$ satisfy the following estimate
%the fixed point equation
% \be\lbl{FPn-x}
% \mu_t^x = \mu_0^x \circ T_{0,t}^x[ W_n,\bar\mu_.] \quad\forall x\in I,
% \ee
% or, equivalently,
% \be\lbl{FPn}
% \bar \mu_t = \bar\mu_0 \circ T_{0,t} [ W_n,\bar\mu_.].
% \ee
% for $\bar\mu_0:=\bar\mu_{n,m,0}$ and $\bar\mu_0:=\bar\mu_{n,0}$ respectively.
%
% Moreover,
\be\lbl{prox-n}
\sup_{t\in\T} \bar d(\bar\mu_{n,m,t},\bar\mu_{n,t})\le 
e^{T} \bar d(\bar\mu_{n,m,0},\bar\mu_{n,0}).
\ee
\end{lem}
\begin{proof}
The proof follows from the fact that 
the empirical measure $\mu_{n,m,t}^x$ and the absolutely continuous measure
$\mu_{n,t}^x$ satisfy 
the fixed point equation
 \be\lbl{FPn-x}
 \mu_t^x = \mu_0^x \circ T_{0,t}^x[ W_n,\bar\mu_.] \quad\mbox{a.e.} \;x\in I,
 \ee
% or, equivalently,
% \be\lbl{FPn}
% \bar \mu_t = \bar\mu_0 \circ T_{0,t} [ W_n,\bar\mu_.].
% \ee
 for $\bar\mu_0:=\bar\mu_{n,m,0}$ and $\bar\mu_0:=\bar\mu_{n,0}$ respectively.

For the continuous measures $\mu_{n,t}^x$ this follows from the proof of 
Theorem~\ref{thm.wellp}. Thus, it remains to verify \eqref{FPn-x}
for the discrete measures $\mu_{n,m,t}^x$. By plugging \eqref{EMn} into \eqref{VFn},
for every $x\in I_{n,k}$ we have 
\be\lbl{write-V}
\begin{split}
V[W_n, \mu_{n,m,.}^x](t,u,x) &= \int_I W_n(x,y) \left\{ \int_\SS D(v-u) 
d\mu_{n,m,t}^y (v)\right\} dy\\
&=\sum_{i=1}^n \int_{I_{n,i}} W_n(x,y) \left\{m^{-1} \sum_{j=1}^m D\left(v_{N,(i-1)m+j}(t)-u\right)
\right\}dy\\
&=N^{-1} \sum_{i=1}^n\sum_{j=1}^m W_{n,ki} D\left(v_{N,(i-1)m+j}(t)-u\right).
\end{split}
\ee
% \begin{rem}
% The right hand side of \eqref{write-V} is constant on each interval $I_{n,i}, \; i\in [n]$.
% Consequently, $T^x_{t,s} [\bar\mu_., W_n]$ is constant in $x$ on $I_{n,i}, \; i\in [n]$.
% \end{rem}
The right hand side of \eqref{write-V} with $u:=u_{N,(k-1)m+l}(t)$ yields the velocity field acting on 
the oscillator $((k-1)m+l)$ of the discrete system \eqref{KMn}. Therefore, by construction of the 
empirical measure \eqref{EMn}, for every $x\in I$ we have
$$
\mu^x_{n,m,t} = \mu^x_{n,m,0} \circ T^x_{0,t} [W_n,\mu_{n,m,.}^x].
$$

Finally, \eqref{prox-n} follows from Lemma~\ref{lem.initial-data}.
\end{proof}

\begin{lem}\lbl{lem.converge-ic}
\be\lbl{converge-ic}
\lim_{m\to\infty} \bar d(\bar\mu_{n,m,0}, \bar\mu_{n,0})=0 \quad
\mbox{almost surely \; (a.s.)}.
\ee
\end{lem}
\begin{proof}
\be\lbl{lets-start}
\bar d(\bar\mu_{n,m,0}, \bar\mu_{n,0}) =\sum_{i=1}^n \int_{I_{n,i}} 
d(\mu^y_{n,m,0}, \mu^y_{n,0}) dy.
\ee

Let $i\in [n]$ and $y\in I_{n,i}$ be arbitrary but fixed. Using the definition 
of the metric $d(\cdot,\cdot)$, we have
\be\lbl{use-the-def}
d(\mu^y_{n,m,0}, \mu^y_{n,0}) = \sup_{f\in\mathcal{L}} 
\left|  \int_\SS f(v) \left(d \mu^y_{n,m,0}(v) -d \mu^y_{n,0}(v)\right) \right|.
\ee

Since $\mathcal{L}$ is a subset of the separable space of continuous functions on $\SS$,
we can choose a countable dense (with respect to $L^\infty$--norm) subset of $\mathcal{L}$,
$\{f_l\}$.

Fix $l\in\N$ and for $y\in I_{n,i}$ consider
\be\lbl{Ynim}
\int f^l(v) d\mu_{n,m,0}^y (v)=
m^{-1} \sum_{j=1}^m   f_l(v^0_{N,im+j})
=:m^{-1}\sum_{j=1}^m   Y^l_{N,im+j}.
\ee
Random variables $Y^l_{N,im+j}, j\in [m],$ are independent and uniformly bounded. 
Further,
\be\lbl{rewrite-f}
\begin{split}
\E\left( m^{-1}\sum_{j=1}^m Y^l_{N,mi+j}\right) & = 
m^{-1}\sum_{j=1}^m \int_\SS  f_l(v) \rho^0(v,x_{N,im+j}) dv\\
&\to \int_{I_{n,i}}\int_\SS f_l(v) \rho^0(v,\xi) dv d\xi, \quad m\to\infty.
\end{split}
\ee

By the Strong Law of Large Numbers,
\be\lbl{SLLN}
\lim_{m\to\infty} m^{-1} \sum_{j=1}^m f_l(v^0_{N,im+j}) = n\int_{I_{n,i}} \int_\SS
f_l(v) \rho^0(v,\xi)dvd\xi \quad \mbox{a.s.}.
\ee
Thus,
$$
\P \left\{ \lim_{m\to\infty} \int_\SS f_l(v) d\mu^y_{n,m,0} (v)=
n\int_{I_{n,i}}\int_\SS f_l (v)\rho^0(v,\xi) dvd\xi\quad \forall l\in N\right\} =1.
$$
By density of $\{f_l\}$ in $\mathcal{L}$, 
$$
\P \left\{ \lim_{m\to\infty} m^{-1} \sum_{j=1}^m f(v^0_{N,im+j}) = n\int_{I_{n,i}} \int_\SS
f(v)\rho^0(v,\xi)dvd\xi \forall f\in \mathcal{L}\right\} =1.
$$
In other words, $\forall i\in [n]$ 
(recall $\mu_{n,m,0}^y$ is constant in $y$ over $I_{n,i}$),
\be\lbl{rephrase}
\lim_{m\to\infty} \sup_{y\in I_{n,i}} d(\mu^y_{n,m,0}, \tilde \mu_{n,i,0})=0\quad \mbox{a.s.}.
\ee
The combination of \eqref{lets-start} and \eqref{rephrase} yields \eqref{converge-ic}.
\end{proof}

\subsection{The main result}
We now turn to the original model \eqref{KM}, \eqref{KM-ic}. In analogy to how it was done
for the auxiliary problem in the previous subsection, for given $n,m\in \N$ and $N=nm,$
we define the empirical measure 
\be\lbl{EM}
\nu_{n,m,t}^x(A)=m^{-1}\sum_{j=1}^m \1_{A} \left(u_{(i-1)m+j}(t)\right), \quad A\in \mathcal{B}(\SS),
\; x\in I_{n,i}, i\in [n],
\ee
where $u_N$ is the solution of the IVP \eqref{KM},\eqref{KM-ic}.
Likewise, using the solution of the IVP for the continuum limit 
\eqref{MF}, \eqref{MF-ic},
we define the continuous measure
\be\lbl{cont-meas}
\nu^x_t (A)=\int_{A} \rho (t,u,x) du, \quad A\in \mathcal{B}(\SS),
\; x\in I.
\ee
The well-posedness of \eqref{MF}, \eqref{MF-ic} on $\T$ is established via the same argument used
for the auxiliary problem in the proof of Theorem~\ref{thm.aux-converge}.

Our goal is to show that for large $n,m\in\N$, the continuos measure \eqref{cont-meas} approximates
the empirical measure \eqref{EM}, and, thus, describes the behavior of solutions of the discrete
model \eqref{KM} with $N=nm$. This is achieved in the following theorem.

\begin{thm}\lbl{thm.main}
For a given $\epsilon>0$
there exist $n_1,m_1\in\N$ such that for $n\ge n_1$, $m\ge m_1$ and
$N=nm$, for the empirical measure \eqref{EM} and the continuous measure
\eqref{cont-meas}, we have
\be\lbl{main}
\sup_{t\in\T}\bar d(\bar \nu_{n,m,t}, \bar\nu_t)<\epsilon
\quad \mbox{a.s.}.
\ee
\end{thm}

\begin{proof}
\begin{enumerate}
\item  Recall that $W_n$ stands for the step function taking on each cell 
$I_{n,i}\times I_{n,j},$ $i,j\in [n],$ a constant value 
$n^2\int_{I_{n,i}\times I_{n,j}} W(x,y)dxdy$, the average value of $W$.
By Lemma~\ref{lem.mart}, 
\be\lbl{by-lem.mart}
\exists n_1\in\N:\quad \|W_n-W\|_{L^2(I^2)} < {\eps\over c_1} \quad \forall n\ge n_1,
\ee
where $c_1:=\left(2(\sqrt{Te^{5T}}+e^T)\right)^{-1}.$ Let $n\ge n_1$ be fixed for the remainder 
of the proof.
\item By Theorem~\ref{thm.aux-converge}, 
\be\lbl{by-aux-converge}
\exists m_1\in \N: \quad
\sup_{t\in\T} \bar d(\bar\mu_{n,m,t}, \bar\mu_{n,t})< {\epsilon\over 2}
\quad \forall m>m_1\; \mbox{a.s.}.
\ee
\item By Lemma~\ref{lem.c-kernel}
\be\lbl{by-lem-c-kernel}
\sup_{t\in\T} \bar d(\bar\mu_{n,t}, \bar\nu_{t})< e^{2T} \|W-W_n\|< 
{e^{2T} \epsilon\over c_1}.
\ee
\item We estimate
\be\lbl{estimate-EM}
\begin{split}
\bar d(\bar\mu_{n,m,t}, \bar\nu_{n,m,t}) &=\int_I 
d( \mu^x_{n,m,t}, \nu^x_{n,m,t}) dx\\
&= n^{-1} \sum_{i=1}^n d( \mu^{x_i}_{n,m,t}, \nu^{x_i}_{n,m,t})\\
&= n^{-1} \sum_{i=1}^n \sup_{f \in\mathcal{L} }
\left| \int_\SS f(v)\left(d \mu^{x_i}_{n,m,t}(v)-d \nu^{x_i}_{n,m,t}(v)\right)\right|\\
&=n^{-1} \sum_{i=1}^n \sup_{f\in\mathcal{L}}
\left| m^{-1}\sum_{j=1}^m
 \left( f(v_{N,(i-1)m+j} (t)) - f(u_{N,(i-1)m+j} (t)) \right)\right|\\
& \le n^{-1} \sum_{i=1}^n \sup_{f\in\mathcal{L}}
m^{-1}\sum_{j=1}^m
\left| f(v_{N,(i-1)m+j} (t)) - f(u_{N,(i-1)m+j} (t)) \right|\\
& \le (nm)^{-1} \sum_{i=1}^n 
\sum_{j=1}^m
\left| v_{N,(i-1)m+j}(t) - u_{N,(i-1)m+j}(t) \right|.
\end{split}
\ee
Further, by the Schwarz inequality followed by the application of 
Lemma~\ref{lem.generic}, we continue the string of estimates in \eqref{estimate-EM}
as follows
\be\lbl{we-continue}
\begin{split}
\bar d(\bar\mu_{n,m,t}, \bar\nu_{n,m,t}) &\le N^{-1}\sum_{j=1}^N
\left| v_{N,j}(t) - u_{N,j}(t) \right|\\
& \le \|v_N-u_N\|_{1,N} 
\le C_1\|W_n-W_N\|_{L^2(I^2)} \\
& \le C_1 \left(\|W_n-W\|_{L^2(I^2)}+\|W-W_N\|_{L^2(I^2)}\right)\\
&\le {2 C_1 \epsilon\over c_1},
\end{split}
\ee
where we used \eqref{by-lem.mart} to derive the last inequality.
\item By combining the estimates in 1.-4., we have
\be\lbl{summing-up}
\bar d(\bar \nu_{n,m,t}, \nu_t) \le 
\bar d(\bar\mu_{n,m,t}, \bar\nu_{n,m,t})+
\bar d(\bar\mu_{n,m,t}, \bar\mu_{n,t})+
\bar d(\bar\mu_{n,m,t}, \bar\nu_{n,m,t})\le \epsilon,
\ee
holding for every $t\in\T$.
\end{enumerate}

\end{proof}

\section{Discussion}\lbl{sec.discuss}

The choice of the KM in this paper was motivated by its role in the theory
of synchronization \cite{Str00} and analytical convenience. 
The analysis in the previous sections can be naturally extended to other 
models. First, by extending the phase space to include $\omega$
it applies  easily to the KM with distributed intrinsic frequencies \eqref{wKM}.
Specifically, let $G=\SS\times\R$ be an extended phase space, 
${\mathcal M}_G$ be the space of Borel
probability measures on $G$, $\bar\mu: x\in I\mapsto \mu^x\in {\mathcal M}_G$
be a ${\mathcal M}_G$-measurable function, and 
${\bar \mu}_t: t\in\T \mapsto {\bar \mu}^x_t$ be a weakly continuous function.
Then we rewrite the equation of characteristics
in the following form:
\be\lbl{ext-CE}
{\p \over \p t} \begin{pmatrix} u\\ \phi\end{pmatrix}
= \begin{pmatrix} \phi + \int_I W(x,y) \left\{ \int_\SS\int_\R
D(v-u) d \mu^y_t (v,\lambda) \right\} dy\\
0 
\end{pmatrix} =: V[W,\bar\mu_.](u,\phi,x,t).
\ee
By assigning initial condition $(u(0),\phi(0))= (u_i^0, \omega_i)$, the right hand side of
the top equation in \eqref{ext-CE} yields the velocity field acting on oscillator $i$ with 
the intrinsic frequency $\omega_i$. Since the newly added second component of the 
vector field in  \eqref{ext-CE} is trivial, all necessary estimates for 
$V[W,\bar\mu_.](u,\phi,x,t)$ are
as before. With the equation of characteristics \eqref{ext-CE} in hand, one can set up 
the fixed point equation and then analyze it exactly in the same way as it was done
in Sections \ref{sec.fixed} and \ref{sec.apply}. We refer an interested reader to 
\cite{ChiMed16}, where the KM with distributed frequencies was analyzed albeit
for Lipschitz graph limits. 

Likewise, there are no principal difficulties in applying our analysis to other models
of interacting dynamical systems on graphs, such as Cucker-Smale model of flocking
\cite{BSB15, MotTad14}, consensus protocols \cite{Med12}, as well as neuronal networks
\cite{BFFT12}. The multidimensionality of the phase space can be handled 
in the same way as explained above for the KM with distributed frequencies.
The treatment of the coupling term, the key ingredient in the analysis, remains the same as
in the present paper. On the other hand, the weighted graph model, which we adopted in
this paper (see \eqref{Xn}-\eqref{edge-set}), provides  a simple unified
treatment of interacting dynamical systems on a variety of deterministic and random graphs,
including Erd\H{o}s-R{\' e}nyi, small-world, and certain approximations of power law graphs
\cite{MedTan17, KVMed17}.

\vskip 0.2cm
\noindent
{\bf Acknowledgements.} GM thanks Carlo Lancellotti for useful discussions.
This work was supported in part by the NSF DMS 1412066 (GM).
\vfill
\newpage

\bibliographystyle{amsplain}
% \bibliography{fubini}

\def\cprime{$'$} \def\cprime{$'$} \def\cprime{$'$}
\providecommand{\bysame}{\leavevmode\hbox to3em{\hrulefill}\thinspace}
\providecommand{\MR}{\relax\ifhmode\unskip\space\fi MR }
\providecommand{\MRhref}[2]{%
  \href{http://www.ams.org/mathscinet-getitem?mr=#1}{#2}
}
\providecommand{\href}[2]{#2}

\end{document}